\documentclass[11pt]{amsart}
\usepackage{amsfonts, amssymb, amsgen, amsbsy, amstext, amsopn, amsmath,
latexsym, indentfirst,color,longtable,verbatim}
\usepackage[english]{babel}
\usepackage[utf8]{inputenc}

\usepackage{verbatim} 
\theoremstyle{plain}

\theoremstyle{definition}

\theoremstyle{remark}

\let\altphi\phi
\let\phi\varphi
\let\varphi\altphi
\let\altphi\undefined
\newcommand{\veps}{\varepsilon}

\newcommand{\bra}[1]{\left(#1\right)}
\newcommand{\cur}[1]{\left\{#1\right\}}
\newcommand{\sqa}[1]{\left[#1\right]}
\newcommand{\ang}[1]{\left<#1\right>}

\newcommand{\abs}[1]{\left\lvert#1\right\rvert}
\newcommand{\norm}[1]{\left\lVert#1\right\rVert}
\newcommand{\Liph}{\operatorname{Lip}}
\renewcommand{\O}{\Omega}

\renewcommand{\H}{\mathbb{H}}

\newcommand{\G}{\mathbb{G}}

\newcommand{\R}{\mathbb{R}}

\newcommand{\cC}{\mathcal{C}}

\newcommand{\cF}{\mathcal{F}}

\newcommand{\cS}{\mathcal{S}}

\newcommand{\cV}{\mathcal{V}}

\newcommand{\ep}{\varepsilon}

\newcommand{\sm}{\setminus}

\newcommand{\lls}{\big(}
\newcommand{\rrs}{\big)}

\renewcommand{\exp}{\mbox{\rm exp}\;\!}

\renewcommand{\span}{\mbox{\rm span}}

\newcommand{\Lie}{\mathrm{Lie}}

\newcommand{\q}{\mathrm q}
\newcommand{\m}{\mathrm m}

\newcommand{\der}{\partial}

\newcommand{\ds}{\displaystyle}

\newcommand{\beqas}{\begin{eqnarray*}}
\newcommand{\eeqas}{\end{eqnarray*}}
\newcommand{\beqa}{\begin{eqnarray}}
\newcommand{\eeqa}{\end{eqnarray}}
\newcommand{\beq}{\begin{equation}}
\newcommand{\eeq}{\end{equation}}
\newcommand{\bce}{\begin{center}}
\newcommand{\ece}{\end{center}}

\newcommand{\pa}[1]{\left( #1 \right)}               
\newcommand{\set}[1]{\left\{ #1 \right\}}            

\newcommand{\qs}[1]{\quad\mbox{ #1} \quad}               

\newtheorem{The}{Theorem}[section]
\newtheorem{Lem}[The]{Lemma}
\newtheorem{Def}[The]{Definition}
\newtheorem{Rem}[The]{Remark}
\newtheorem{Pro}[The]{Proposition}
\newtheorem{Cor}[The]{Corollary}
\newtheorem{Exa}[The]{Example}

\newcommand{\bt}{\begin{The}}
\newcommand{\et}{\end{The}}
\newcommand{\bl}{\begin{Lem}}
\newcommand{\el}{\end{Lem}}
\newcommand{\bd}{\begin{Def}\rm}
\newcommand{\ed}{\end{Def}}
\newcommand{\br}{\begin{Rem}\rm}
\newcommand{\er}{\end{Rem}}
\newcommand{\bpr}{\begin{Pro}}
\newcommand{\epr}{\end{Pro}}
\newcommand{\bc}{\begin{Cor}}
\newcommand{\ec}{\end{Cor}}
\newcommand{\bex}{\begin{Exa}}
\newcommand{\eex}{\end{Exa}}

\begin{document}

\title
{On Lipschitz vector fields and the Cauchy problem in homogeneous groups}
\author{Valentino Magnani}
\address{Valentino Magnani, Dipartimento di Matematica, Universit\`a di Pisa \\
Largo Bruno Pontecorvo 5 \\ I-56127, Pisa}
\email{valentino.magnani@unipi.it}
\date{\today}

\author{Dario Trevisan}
\address{Dario Trevisan, Dipartimento di Matematica, Universit\`a di Pisa \\
Largo Bruno Pontecorvo 5 \\ I-56127, Pisa}
\email{dario.trevisan@unipi.it}
\date{\today}

\thanks{This work is supported by the University of Pisa, Project\ PRA\_2016\_41}
\date{\today}

\begin{abstract}
We introduce a class of ``Lipschitz horizontal'' vector fields in homogeneous groups, for which we show equivalent descriptions, e.g.\ in terms of suitable derivations. We then investigate the associated Cauchy problem, providing a uniqueness result both at equilibrium points and 
for vector fields of an involutive submodule of Lipschitz horizontal vector fields. We finally exhibit a counterexample to the general well-posedness theory for Lipschitz horizontal vector fields, in contrast with the Euclidean theory.
\end{abstract}

\maketitle


\section{Introduction}\label{sec:intro}

Homogeneous groups provide a well-established setting for the study of second-order sub-elliptic operators \cite{bonfiglioli_stratified_2007, folland_hardy_1982}. On the contrary, first order operators (vector fields) on homogeneous groups and their flows are less studied from an intrinsic point of view.
Our starting point is a rather general question: to what extent does an analogue of the Cauchy-Lipschitz theory for ordinary differential equations (ODE) holds true in the sub-Riemannian framework?

Aim of this paper is to single out a natural class of ``intrinsically Lipschitz vector fields''
in homogeneous groups, and then to investigate the corresponding Cauchy problem
\beq\label{eq:ODE0}
 \dot \gamma (t) = b(\gamma(t)), \quad \text{$t \in [0,T]$}
\eeq
associated to such vector fields and a given initial datum.
To simplify our discussion, let us consider the Heisenberg group $\H$, identified with $\R^3$
and equipped with the group operation
\begin{equation}\label{eq:heisenberg-operation}(x_1, x_2, x_3) (y_1, y_2, y_3) =  (x_1+y_1, x_2+y_2, x_3+y_3 +(x_1y_2-x_2y_1)). \end{equation}
We consider the distance $d(x,y) = \norm{x^{-1}y}$, where
\begin{equation}\label{eq:heisenberg-norm}\norm{ \bra{x_1, x_2, x_3} } :=  \bra{ \bra{|x_1|^2 + |x_2|^2}^2 + |x_3|^2 }^{1/4} \end{equation}
and fix the two horizontal vector fields
\begin{equation}\label{eq:heisenberg-fields} X_1(x_1, x_2, x_3)= \der_1 - x_2\der_3, \quad  X_2(x_1, x_2, x_3)= \der_2+ x_1\der_3.\end{equation}
In this setting, a ``Lipschitz horizontal'' vector field $b$ on $\H$ is a mapping
\[
x \mapsto b(x) = a_1 (x) X_1(x)+ a_2 (x) X_2(x),
\]
where $a_1$, $a_2$ are Lipschitz functions on $\H$, with respect to $d$.
The difficulty of the Cauchy problem \eqref{eq:ODE0} arises from little regularity of $b$,
since $d$ is not bi-Lipschitz equivalent to the Euclidean distance at any scale.
Clearly, the continuity of $b$ ensures the existence of solutions, so our issue concerns the uniqueness problem.

On the other hand, $d$ is compatible with the algebraic structure of the group, since
it is left invariant and 1-homogeneous with respect to suitable group dilations,
see Section~\ref{sec:homgauges}. 
It can be easily shown that $d$ is bi-Lipschitz equivalent to the celebrated sub-Riemannian distance on $\H$.

In sum, from the geometry of the group two competing aspects appear. 
On the positive side, the horizontality forces solutions to move along paths where their restriction is more regular. On the negative side, Lipschitz horizontal vector fields are not Lipschitz in the Euclidean sense, hence most of the classical arguments fail to apply. For these reasons, the uniqueness problem in this setting does not seem to have a natural answer.

We will present both positive and negative results on the uniqueness problem. In the Heisenberg group, we present a counterexample to uniqueness for the Cauchy problem with respect to a Lipschitz horizontal vector field, see Section~\ref{sec:counterexample}. In other words, the ODE admits two different solutions starting from the same initial position. However, the picture is more interesting, since two uniqueness results can be proved in the more general framework of homogeneous groups, see Section~\ref{sec:NotBasic}.

The first one considers equilibrium points, see Section~\ref{sec:equilibrium}.
When $\bar{x}$ is an equilibrium point for a vector field $b$ in a homogeneous group and 
$\gamma$ is another solution to \eqref{eq:ODE0}, then the function $d(\bar{x}, \gamma(t))$ satisfies a differential inequality, which by Gronwall lemma entails uniqueness and a ``quantitative'' stability, see Theorem~\ref{thm:uniqueness-stationary} for a precise statement. Indeed, in this case we can control the left-translated curve $t \mapsto \bar{x}^{-1} \gamma(t)$, since it is driven by a Lipschitz horizontal vector field (Lemma~\ref{lem:change-coordinates-counterexample}). A similar argument, where $\bar x$ is replaced by a nonconstant solution $\Gamma$
would not apply, because the difference of two non constant solutions $t  \mapsto \Gamma(t)^{-1}\gamma(t)$ with respect to the group operation would be no longer a horizontal curve, in general, as the vector fields $X_1$, $X_2$ may not be right invariant. 
In other words, the naive search for contraction estimates, as in the Euclidean case, fails.  

Our second uniqueness result confirms the above mentioned positive side.
When the vector field, besides its horizontality, is a combination of commuting directions,
then the ``regularizing effect'' of these directions prevails.
Precisely, if the vector field belongs to a special ``involutive'' submodule of 
Lipschitz horizontal vector fields, then solutions belong to a linear subspace that is
bi-Lipschitz equivalent to an Euclidean space. Once this is proved, then uniqueness follows 
by Euclidean uniqueness, see Theorem~\ref{thm:uniqueness-involutive}. 

To put our investigation in a wider perspective, let  us point out that one of the problems which stimulated the present work is originated after the article \cite{ambrosio_well-posedness_2014}, where L.~Ambrosio and the second author developed a theory of ordinary differential equations and flows in abstract metric measure spaces, in the spirit of DiPerna-Lions \cite{diperna_ordinary_1989}, i.e.\ for Sobolev vector fields. Carnot groups seen as metric measure spaces, endowed with an invariant distance and an invariant measure, fit well in the framework considered in \cite{ambrosio_well-posedness_2014}, in view of the regularizing effect of the subelliptic heat semigroup, see e.g.\ \cite[Theorem 4.10]{folland_1975}. Hence it is natural to investigate whether new well-posedness results may follow from such an abstract and general theory. Unfortunately, this is not the case, since the notion of ``abstract'' Sobolev vector field introduced in \cite{ambrosio_well-posedness_2014} appears to be stronger than the ``natural'' one, i.e.\ requiring all coefficients to belong to some intrinsic Sobolev space. For this reasons, the analogous of DiPerna-Lions theories for such Sobolev vector fields on sub-Riemannian spaces (e.g., on the Heisenberg group) remains an open problem. For instance, it would be interesting to see if the ideas appearing in known counterexamples in the Euclidean theory \cite{aizenman_vector_1978, depauw_non_2003} could be adapted to entail non-uniqueness for sub-Riemannian DiPerna-Lions flows.

Finally, it is also interesting to mention that ODEs with low regularity appear in
connection with geometric problems like the characterization of intrinsic Lipschitz graphs in the Heisenberg group. In \cite{bigolin_intrinsic_2015} the functions defining such graphs are proved to be characterized as solutions of a suitable first order nonlinear system of PDEs.
An important aspect of \cite{bigolin_intrinsic_2015} is the construction of a suitable ``Lagrangian parametrization'' for singular vector fields, that also differs from DiPerna-Lions type flows.

\medskip

{\bf Acknowledgements.} 
The authors thank Luigi Ambrosio for a fruitful discussion on the
uniqueness problem for commutative Lipschitz horizontal vector fields. 
They also thank Eugene Stepanov for pointing out the useful reference \cite{rachocircunkova_solvability_2008}.
 
%
%
%
%
%
%
%
%
%
%
%
%
\section{Notation and basic facts}\label{sec:NotBasic}

We recall some basic facts about homogeneous groups, see for instance \cite[Chapter 2.A-C]{folland_hardy_1982}.
%
%
%
%
%

\subsection{Homogeneous groups}\label{sec:homgen}
A {\em homogeneous group} can be seen as a graded linear space $\G=H^1\oplus\cdots\oplus H^m$ equipped with a polynomial group operation such that its Lie algebra $\Lie(\G)$ is {\em graded}. This assumption corresponds to the following conditions
\beq\label{eq:LieG}
\Lie(\G)=\cV_1\oplus\cdots\oplus\cV_\iota, \qquad [\cV_i,\cV_j]\subset \cV_{i+j}
\eeq
for all integers $i,j\ge0$ and $\cV_{j}=\{0\}$ for all $j>\iota$, with $\cV_\iota\neq\{0\}$. 
The integer $\iota \ge 1$ is the {\em step} of the group.
In any homogeneous group we can introduce intrinsic dilations $\delta_r:\G\to\G$
that are compatible with its algebraic structure. These are linear mappings with the additional condition 
\[
\delta_r(x)=r^ix 
\]
for each $x\in H^i$, $r>0$ and $i=1,\ldots,\iota$.

Identifying $\G$ with the tangent space $T_0\G$ of $\G$ at the origin $0$,
we have a canonical isomorphism between $H^j$ and $\cV_j$, that associates to each
$v\in H^j$ the unique left-invariant vector field $X_v\in\cV_j$ such that $X_v(0)=v$. 
It is also convenient to assume that $\G$ is equipped with a Lie product that induces on
$\G$ a structure of Lie algebra and its group operation is given by the Baker-Campbell-Hausdorff formula.
These additional conditions are possible since the exponential mapping 
\[
\exp:\Lie(\G)\to\G
\]
under our assumption is a bianalytic diffeomorphism, as we have for any simply connected nilpotent Lie group $\G$.

We denote by $\q$ the dimension of $\G$, seen as a linear space. A {\em graded basis} $(e_1,\ldots,e_\q)$ of $\G$ is a basis of vectors such that 
\[
(e_{\m_{j-1}+1},e_{\m_{j-1}+2},\ldots,e_{\m_j})
\]
is a basis of $H^j$ for each $j=1,\ldots,\iota$, where $\m_j=\sum_{i=1}^j\dim H^i$
for every $j=1,\ldots,\iota$. We also set $\m_0=0$ and $\m=\m_1$. For each $i\in\set{1,\ldots,\q}$ we define its \emph{degree} $d_i$ as the unique $d_i \in \set{1,\dots,\iota}$  such that
\[
\m_{d_i-1}< i\leq \m_{d_i}.
\]
A choice of such graded basis provides the corresponding {\em graded coordinates} 
\[
x=(x_1,\ldots,x_\q)\in\R^\q,
\]
giving the unique element $\sum_{i=1}^\q x_ie_i\in\G$. In this basis, the group operation is a polynomial map \cite[Proposition~1.2]{folland_hardy_1982}.
 We fix throughout a graded basis and associated coordinates.
Let us point out that dilations in terms of graded coordinates read as follows:
\[
\delta_r\bra{\sum_{i=1}^\q x_ie_i} = \sum_{i=1}^\q r^{d_i} x_ie_i.
\] 

\subsection{Homogeneous gauges}\label{sec:homgauges}
For $\lambda \in \R$, we say that a function $f: \G\setminus\{0\}  \to \R$ is $\lambda$-{\em homogeneous} (or homogeneous of degree $\lambda$) if one has $f(\delta_r x) = r^\lambda f(x)$, for every $r>0$, $x \in \G\setminus \{0\}$. A {\em homogeneous gauge} is any continuous $1$-homogeneous function $\norm{\cdot}: \G \to [0, \infty)$ such that $\norm{\cdot}=0$ if and only if $x =0$ and $\|{x^{-1}}\| = \|{x}\|$, for every $x \in \G$. 
If $f: \G\sm\set{0} \to \R$ is a continuous homogeneous function of degree $\lambda$ and $\norm{\cdot}$ is a homogenous gauge, then for some constant $C\ge 0$, there holds
\begin{equation}
\label{eq:domination-homogeneous-functions} \abs{f(x) } \le C \norm{ x}^\lambda \quad \text{for every $x \in \G\sm\set{0}$.}
\end{equation}
Indeed, the function $x \mapsto \abs{f(x)}/ \norm{x}^\lambda$ is bounded on the compact set $\{x \, : \, \norm{x} =1\}$ and $0$-homogeneous on all $\G\sm\set{0}$. In particular, any two homogeneous gauges are equivalent, i.e., if $\norm{\cdot}$ and $\norm{\cdot}'$ are homogenous gauges, then for some constant $C \ge1$ one has
\begin{equation}\label{eq:homogeneous-gauge-equivalence} C^{-1} \norm{ x } \le \norm{x}' \le C \norm{x},\quad \text{for every $x \in \G$.}
\end{equation} 
We fix throughout the following homogeneous gauge (in terms of the graded coordinates)
\begin{equation}\label{eq:smooth-gauge} \norm{x} := \bra{ \sum_{i=1}^\q \bra{x_i}^{ \rho_i} }^{1/N} \text{for $x \in \G$,}
\end{equation}
where $N$ is a fixed integer such that $\rho_i:= N/d_i$ is an even natural number for every $i$, e.g., $N = 2 (\iota !)$. This choice entails that $x \mapsto \norm{x}^N$  is smooth.

One can prove \cite[Proposition~1.26]{folland_hardy_1982} that any left invariant vector field $X_i$ with $X_i(0) = e_i \in H^{d_i}$ can be explicitly written  as
\beq\label{eq:explicit-canonical-fields-2} 
X_i (x) = \sum_{j=1}^\q p_i^j (x) \partial_j, \quad \text{for every $x\in \G$.}
\eeq
where $p_i^j (x)$ satisfies $p_i^j= \delta_i^j$ if $d_j \le d_i$. 
When $d_j>d_i$, $p_{ij}$ is a homogeneous polynomial of degree $d_j-d_i$ with respect to
the group dilations $\delta_r$. Precisely, with respect to our fixed graded coordinates $(x_1, \ldots, x_q)$
one has
\[
p_i^j( \delta_r (x) )  = r^{d_j -d_i} p_i^j( x ),\quad \text{for every $r \ge 0$ and $x \in \G$.}
\]
It follows from \eqref{eq:domination-homogeneous-functions} that there exists some constant $C \ge1$ such  that
\beq\label{eq:estimate-polinomials}
| p_i^j( x ) | \le C \norm{ x }^{d_j - d_i}, \quad \text{for every $x \in \G$}
\eeq
whenever $d_j\ge d_i$.
%
A {\em homogeneous distance} $d$ on $\G$ is a left invariant distance with the property $d(\delta_rx,\delta_ry)=r\,d(x,y)$ for all $x,y\in\G$ and $r>0$. 
Then $\G \ni x \mapsto d(x,0)$ is a homogeneous gauge.

\subsection{Horizontal subbundle and horizontal gradient}

We fix a scalar product on $\G$ such that the fixed graded basis is orthonormal. 
If $X_i\in\Lie(\G)$ is the unique left invariant vector field with $X_i(0)=e_i$, we thus have 
a
basis 
$(X_1,\ldots,X_\q)$  of $\Lie(\G)$. 

The so-called {\em horizontal directions} are defined by taking the
first layer $\cV_1$ of $\Lie(\G)$ and this defines the {\em horizontal fiber}
\[
H_x\G=\{X(x)\in T_x\G: X\in\cV_1\}
\]
as $x$ varies in $\G$.
We denote by $H\G$ the {\em horizontal subbundle} of $\G$, whose fibers are $H_x\G$.
We denote by $|\cdot|$ the norm arising from the fixed scalar product of $\G$.

Given an open set $\O\subset\G$, we denote by $H\O$ the restriction of the {\em horizontal subbundle} $H\G$ to the open set $\O$, whose {\em horizontal fibers} $H_x\G$ are restricted to all points $x\in \O$. 
We also introduce the natural class of $C^1$ smooth functions with respect to 
the horizontal subbundle. A function $u:\Omega\to\R$ belongs to $C^1_H(\Omega)$ if its
pointwise {\em horizontal gradient} 
\[
\nabla_H u:=(X_1 u, X_2u, \ldots, X_\m u ) = \sum_{i =1}^m (X_i u) X_i \]
is continuous in $\Omega$. Let us point out that there exist $C^1_H$ functions that are not differentiable on a set of positive measure, \cite{magnani_coarea_2005}. On the other hand, smooth functions
are dense in $C^1_H(\Omega)$ in the natural topology, \cite{FSSC4, FSSC1, garofalo1996}.


%
%
%

\section{Horizontal vector fields}

In this section, we study equivalence between alternative definitions of horizontal vector fields
that are Lipschitz continuous with respect to a homogeneous distance.
Before we reduce to horizontal fields, we state some general facts about time-dependent vector fields and associated ODE's.

\subsection{Vector fields and ODE's}

Let $(b_t)_{t \in [0,T]}$ be a continuous time-dependent vector field on an open set $\O \subseteq \G$, i.e. a map of the form
\begin{equation}
\label{eq:def-general-b} b(t,x)  = \sum_{i=1}^\q a_i(t,x) X_i(x), \quad \text{for $(t,x) \in [0,T] \times \O$,}\end{equation}
where $a_i$ is a continuous function for all $i \in \{1, \ldots, \q\}$.
As usual, we define a solution to the following ODE 
\begin{equation}\label{eq:ODE}\dot \gamma (t) = b_t(\gamma(t)) \quad \text{for $t \in [0,T]$,}  \end{equation}
as a continuous curve $\gamma:[0,T]\to \O$ such that the identity
\[
\gamma(t)=\gamma(0) + \int_{0}^t b(s,\gamma(s))\,ds \quad \text{holds for every $t \in [0,T]$.}
\]
Given $x_0 \in \Omega$, the associated Cauchy problem reads as
\begin{equation}\label{eq:cauchy-problem}
\left\{ \begin{array}{ll}
\dot \gamma (t) = b_t(\gamma(t)) & \text{for $t \in [0,T]$,} \\
\gamma(0) = x_0. &
\end{array}\right.
\end{equation}
Taking into account \eqref{eq:explicit-canonical-fields-2}, the ODE \eqref{eq:ODE} can be written in terms of the graded coordinates $(\gamma_i)_{i=1}^\q$ as the following system of differential equations,
\begin{equation}
\label{eq:ode-system-componentwise}
\dot \gamma_i(t)= \sum_{j=1}^\q a_j (t,\gamma(t)) p^{i}_j(\gamma(t)), \quad \text{$t \in (0,T)$,}
\end{equation}
for all $i \in \{1, \ldots, \q \}$. 
The following lemma, exploiting the left-invariance of the vector fields $X_i$, 
will be useful to translate the initial position in the Cauchy problem \eqref{eq:ODE} 
to the origin, when $b$ is given by \eqref{eq:def-general-b}.

\begin{Lem}\label{lem:change-initial-point-general}
Let $(b_t)_{t \in [0,T]}$ be a continuous vector field as in \eqref{eq:def-general-b}, let $(\gamma_t)_{t \in [0,T]}$ solve \eqref{eq:ODE} and let $\bar{x} \in \Omega$. Then the curve $t \mapsto \bar{\gamma}(t):= \bar{x} \gamma$ solves the ODE \eqref{eq:ODE} with respect to the vector field 
\[
 \bar{b}(t,x) = \sum_{i=1}^\q a_i(t, \bar{x}^{-1} x) X_i(x) \quad \text{where $(t,x) \in [0,T] \times (\bar{x} \O)$.}
\]
\end{Lem}

\begin{proof}
Since the operation $y \mapsto \bar{x} y = \ell_{\bar{x}} (y)$ is smooth, then for $t \in (0,T)$, one has 
\begin{equation*}\begin{split}
 \frac{d}{dt} \bar{\gamma}(t) & = \frac{d}{dt} \ell_{\bar{x}} (\gamma(t)) = d \ell_{\bar{x}}|_{\gamma(t)} \bra{  \frac{d}{dt} \gamma(t)} \\
 & = d \ell_{\bar{x}}|_{\gamma(t)} \bra{ \sum_{i=1}^\q a_i(t,\gamma(t)) X_i(\gamma(t))} \\
 & = \sum_{i=1}^\q a_i(t,\gamma(t)) d \ell_{\bar{x}}|_{\gamma(t)} \bra{ X_i(\gamma(t))} \\
 & =  \sum_{i=1}^\q a_i(t,\bar{x}^{-1} \bar{x} \gamma(t)) X_i(\bar{x}\gamma(t)) = \bar{b}(t,\bar{\gamma}(t)). \qedhere
 \end{split}
 \end{equation*}
\end{proof}

\subsection{Continuous horizontal vector fields}

Given an open set $\Omega \subseteq \G$, we denote by $C(H\O)$ the linear space of all continuous sections of $H
\O$, i.e., those maps $b$ on $\O$ which can be represented as
\begin{equation}\label{eq:horizontal-vector-field}
b(x)= \sum_{i=1}^\m a_i(x) X_i(x), \quad \text{for $x \in \O$,}
\end{equation}
for continuous functions $a_i \in C(\O)$ and $i \in \{1, \ldots,\m\}$.
As an example, we may consider $V \in C^1_H(\O)$ and the vector field corresponding 
to the horizontal gradient
\[
 \sum_{i =1}^\m (X_i V) X_i \in C(H\O).
\]
An equivalent description of vector fields is provided by derivations, which are a slightly more intrinsic notion, with the advantage of allowing for generalizations to abstract settings, see e.g.\ \cite{ambrosio_currents_2000, weaver_lipschitz_2000}.
It is sufficient to notice that any  $b \in C(H\O)$ naturally induces a linear map 
$D_b: C^\infty(\O) \to C(\O)$, $f \mapsto D_b(f) := \sum_{i=1}^{m_1} a_i X_i f$.
This special form of $D_b$ satisfies the following estimate
\begin{equation}\label{eq:derivation-continuity}
\abs{D(f)} (x) \le h (x) \abs{\nabla_H f}(x)  \quad \text{for every $x \in \O$ and $f \in C^\infty(\O)$,}
\end{equation}
with $D=D_b$ and $h(x) = \sqrt{ \sum_{i=1}^{m_1} \abs{a_i}^2(x)}$. Let us remark that linearity and inequality \eqref{eq:derivation-continuity} easily entail the  usual calculus rules. Any linear mapping $D:C^\infty(\O)\to C(\O)$ for which \eqref{eq:derivation-continuity} holds, also satisifes the Leibniz rule: it is sufficient to optimize the choice of constants $\alpha$, $\beta \in \R$ in the inequality
\begin{equation*}\begin{split} \abs{ D(fg)(x) - D(f)(x) \alpha - \beta D(g)(x) }  &  = \abs{ D(fg - f \alpha - \beta g)(x)} \\
& \le h(x)\abs{\nabla_H \bra{fg - f \alpha - \beta g  } }(x) \\
& \le   h(x)\sqa{ \abs{\nabla_H f }(x) \abs{ g(x)- \alpha} +   \abs{ f(x)- \beta}\abs{\nabla_H g }(x)}.
\end{split}\end{equation*}
Similarly, the chain rule holds, for $k \ge 1$, $F \in C^\infty( \R^k)$, $f_1, \ldots, f_k \in C^\infty(\O)$:
\begin{equation}\label{eq:chain-rule} D( F(f_1, \ldots, f_k) ) (x) =  \sum_{i=1}^k \partial_i F (f_1(x), \ldots, f_k(x)) D(f_i)(x), \quad \text{for every $x \in \O$.} \end{equation}
%
The assumption that $D_b$ takes values in $C(\O)$ is  motivated by the fact that we look for  vector fields defined at \emph{every} point $x \in \O$ (and not, for instance, at a.e.\ point with respect to Lebesgue measure). According to the following lemma, derivations provide an equivalent definition
of Lipschitz horizontal vector fields.

\begin{Lem}\label{lem:derivations-vector-fields}
Let $D: C^\infty(\O) \to C(\O)$ be a linear mapping that satisfies \eqref{eq:derivation-continuity}, for some function $h: \O \to [0, \infty)$. Then, there exists a unique $b \in C(H\O)$  such that $D$ is induced by $b$, that is $D=D_b$. 
\end{Lem}

\begin{proof}
Given $f \in C^\infty(\O)$, by \eqref{eq:chain-rule} with $k = \q$, $F=f$ and $f_i (x)=\eta_i(x) =  x_i$, we obtain the representation $D(f)(x)  = \sum_{i=1}^\q \partial_i f (x) a_i(x)$, for $x \in \O$, with $a_i(x) := D(\eta_i)(x)$,   $i \in \{1, \ldots, \q \}$. By (pointwise) inverting the linear identity \eqref{eq:explicit-canonical-fields-2},  there exists (continuous) functions $b_i \in C(\O)$ such that
\begin{equation}\label{eq:vector-field-temp-representation} D(f)(x)  = \sum_{i=1}^\q X_i f (x) b_i(x),\quad \text{for $x \in \O$.}\end{equation}

To show that \eqref{eq:horizontal-vector-field} holds, it is sufficient to prove that $b_i= 0$ for every $i \in \cur{m_1+1, \ldots, \q}$.  Let us fix $\bar{x} \in \O$ and introduce the function $x \mapsto \bar{\eta}_j(x) := \eta_j( {\bar{x}^{-1}} x)$ so that, by left invariance, there holds
\[ X_i ( \bar{\eta}_j) (\bar{x}) = X_i(\eta_j \circ \ell_{\bar{x}^{-1}} ) (\bar{x}) =  X_i ( \eta_j ) (0) =\delta_{ij},\quad \text{for  $i,j \in \cur{1, \ldots, \q}$.}\]
In particular, from \eqref{eq:vector-field-temp-representation} we have $b_i(\bar{x}) = D(\bar{\eta_i})(\bar{x})$ and \eqref{eq:derivation-continuity} yields
\[ \abs{D(\bar{\eta}_i) (\bar{x})} \le h(\bar x ) \abs{ \nabla _H \bar \eta_i } (\bar{x}) = 0, \quad \text{for $i \in \cur{m_1+1, \ldots, \q}$.}\qedhere\]
\end{proof}

Following a more algebraic viepoint, a third equivalent presentation of $C(H\O)$ is 
to see it as the $C(\O)$-module generated by the vector fields $X_1, \ldots, X_\m$ (or equivalently, by the class of the smooth horizontal vector fields).

\bd
If $R$ is a ring of real-valued functions on an open set $\Omega\subset\G$ 
and $V$ is a linear space of vector fields on $\Omega$, where $\G$ is a homogeneous
group, then the $R$-module generated by $V$ is the smallest $R$-module containing
$V$, namely
\[
\left\langle V\right\rangle_R=\bigcap_{M\in\cF} M,
\]
where $\cF$ is the family of $R$-modules containing $V$. 
\ed
\br
It is easy to realize that $C(H\Omega)=\langle \cV_1\rangle_{C(\O)}$,
where $\cV_1$ is the linear space of left invariant horizontal vector fields,
see \eqref{eq:LieG}.
\er

\subsection{Lipschitz horizontal vector fields}\label{subs:horizLip}

If $\O$ is an open subset of $\G$, we denote by $\Liph(\O)$ the space
of real-valued functions defined on $\O$ that are Lipschitz continuous with respect to
a fixed homogeneous distance $d$ of $\G$. The choice of the distance is clearly
immaterial, since obviously all homogeneous distances are bi-Lipschitz equivalent, 
see \eqref{eq:homogeneous-gauge-equivalence}.

We define a {\em Lipschitz horizontal vector field} as a vector field $b \in C(H\O)$ 
such that, in the expression \eqref{eq:horizontal-vector-field},
$a_i \in \Liph(\O)$ for any $i\in \{1, \ldots, \m\}$. We then write $b \in \Liph(H\O)$.
The fact that this is an intrinsic notion is shown by the following lemma,
that provides an equivalent definition, in terms of derivations.
\begin{Lem}
Let $D: C^\infty(\O) \to \Liph(\O)$ be linear and satisfy \eqref{eq:derivation-continuity}, for every $f \in \cC^\infty(\O)$ and some function $h: \O \to [0, \infty)$. Then, there exists a unique $b \in \Liph(H\O)$ such that $D$ is induced by $b$ i.e.\ $D=D_b$. 
\end{Lem}

\begin{proof}
It is sufficient to notice that existence of some $b \in \cC(H\O)$ is provided by Lemma~\ref{lem:derivations-vector-fields},  and since $a_i = D(\eta_i) \in \Liph(\O)$, we conclude that $b \in \Liph(H\O)$.
\end{proof}
\br
As for continuous horizontal vector fields, we have $\Liph(H\O)=\langle\cV_1\rangle_{\Liph(\O)}$.
\er 


\begin{Def}[Time-dependent fields]
A {\em time-dependent Lipschitz horizontal vector field} is a mapping $(b_t)_{t \in [0,T]}$ with $b_t \in \cC(H\O)$, for $t \in [0,T]$, where defining 
\[
b_t(x)= \sum_{i=1}^{\m} a_i(t,x) X_i(x) \quad \text{for $t \in [0,T]$ and $x \in \O$},
\]
then $a_i$ is continuous on $[0,T]\times \O$ and there exists a constant $C>0$ such that
\[
\abs{a_i(t,x) - a_i(t,y)} \le C\, d(x,y) \quad \text{for every $i\in \{1, \ldots,\m\}$,\ $x,y\in \O$ and $t \in [0,T]$},
\]
where $d$ is a homogeneous distance.
\ed

%

\section{Uniqueness results}

\subsection{Equilibrium points}\label{sec:equilibrium}

In this section, we provide a result concerning well-posedness for the Cauchy problem associated to \eqref{eq:ODE}, when $\bar x$ is an equilibrium point for $b$, i.e.\ $b(t, \bar x)=0$ for $t \in [0,T]$.
Our argument applies to the general class of vector fields $b$ as in \eqref{eq:def-general-b},
therefore not necessarily Lipschitz nor horizontal, provided that
\begin{equation}
\label{eq:degeneracy-graded-order}
\sum_{i=1}^\q |a_i(t,x)|^{1/d_i} \le c_t d\bra{x, \bar{x}} \quad \text{for every $x \in \Omega$, a.e.\ $t \in (0,T)$,}
\end{equation}
holds for some $(c_t)_{t \in (0,T)} \in L^1(0,T)$. Notice that we may equivalently replace the left hand side above with $\norm{ \sum_{i=1}^q a_i(t,x) e_i}$, i.e.\ the norm of the vector field (slightly improperly) seen as a point in $\G$.

\begin{The}\label{thm:uniqueness-stationary}
Let $b: [0,T]\times \O \to \G$ be a vector field in the form $b(t,x)= \sum_{i=1}^\q a_i(t,x) X_i(x)$, such that \eqref{eq:degeneracy-graded-order} holds for some $(c_t)_{t \in (0,T)} \in L^1(0,T)$. Then, there exists constants $c>0$ such that, for every solution $(\gamma(t))_{t \in [0,T]}$ to \eqref{eq:ODE}, one has
\begin{equation}\label{eq:stability} d( \gamma(t), \bar x ) \le c e^{c\int_0^T c_s ds} d(\gamma(0), \bar x ),\quad  \text{for every $t \in [0,T]$.}\end{equation}
In particular, the constant solution $\gamma(t) = \bar{x}$ to \eqref{eq:ODE} is unique and, if $(\gamma^n(t))_{t \in [0,T]}$ are solutions to \eqref{eq:ODE} such that $\gamma^n(0) \to \bar{x}$, as $n \to \infty$, then they converge uniformly on $[0,T]$ to the constant solution $\bar{x}$.
\end{The}

The following corollary is then straightforward.

\begin{Cor}
Let $b \in \Liph(H\O)$, $\bar{x} \in \O$ satisfy $b(t,\bar{x})=0$ for every $t \in [0,T]$. Then, the solution to the Cauchy problem \eqref{eq:cauchy-problem}, with $x_0 = \bar{x}$, is unique.
\end{Cor}

\begin{proof}
By Lemma~\ref{lem:change-initial-point-general} with $\bar{x}^{-1}$ in place of $\bar{x}$, we may assume that $\bar{x} = 0$: indeed, condition \eqref{eq:degeneracy-graded-order} for the ``translated'' vector field $\bar{b}$ reads as
\[ \sum_{i=1}^\q a_i(t, \bar{x} x )^{1/d_i} \le c_t d( x,0), \quad \text{for every $x \in \bar{x}^{-1} \O$,}\]
which follows at once by left invariance of the distance $d$.
By the equivalence of homogeneous gauges~\eqref{eq:homogeneous-gauge-equivalence}, it is sufficient to show \eqref{eq:stability} with the homogeneous gauge $\norm{\cdot}$ of \eqref{eq:smooth-gauge} in place of $d(\cdot, 0)$. 
  Actually, taking $N$-th powers of both sides, we have to prove the inequality
\begin{equation}
\label{eq:grownall}
 \sum_{i=1}^\q \bra{\gamma_i(t)}^{\rho_i}  \le c\,e^{c\int_0^tc_s ds}  \sum_{i=1}^\q \bra{\gamma_i(0)}^{\rho_i},  
\quad \text{ for $t \in [0,T]$,} \end{equation}
where we recall the notation $\rho_i := N/d_i$ (which by construction it is an even integer). 
For $i \in \{1, \ldots, q\}$, by \eqref{eq:ode-system-componentwise}, we have
\[ \frac {d}{dt} \bra{\gamma_i(t)}^{ \rho_i } = \rho_i \bra{ \gamma_i(t)}^{\rho_i -1} \dot{\gamma}_i(t) = \rho_i \bra{ \gamma_i(t)}^{\rho_i -1} \sum_{j=1}^{q} a_j(t,\gamma(t)) p^{i}_j(\gamma(t) ),\]
where $p^{i}_j(x)$ satisfies \eqref{eq:estimate-polinomials}. 
For every $j \in \{1, \ldots, \q\}$, we have the following estimates 
\[\abs{ \bra{ \gamma_i(t)}^{\rho_i -1} a_j(t,\gamma(t)) p^{i}_j(\gamma(t) ) } \le
\norm{\gamma(t)}^{N (\rho_i-1)/\rho_i} \abs{  a_j(t,\gamma(t)) p^{i}_j(\gamma(t) ) }.\]
By \eqref{eq:estimate-polinomials} and \eqref{eq:degeneracy-graded-order}, for some constant $c'> 0$, we get
\[  \abs{ a_j(t,\gamma(t)) p^{i}_j(x(t)) } \le c' c_t\norm{\gamma(t)}^{{d_j+ (d_i-d_j)}} = \norm{\gamma(t)}^{N/\rho_i}.\]
Therefore, we conclude that for some (possibly larger) constant $c'' > 0$, there holds
\[ {  \frac {d}{dt} \bra{\gamma_i(t)}^{ \rho_i } } \le c'' c_t \norm{\gamma(t)}^N, \quad \text{for $t \in (0,T)$.}\]
Summing over $i\in \{1, \ldots, \q\}$, we have
\[ {  \frac {d}{dt} \norm{ \gamma(t) }^N } \le  c''' c_t \norm{\gamma(t)}^N, \quad \text{ for $t \in (0,T)$,}\]
and Gronwall inequality entails \eqref{eq:grownall}.
\end{proof}

\subsection{Involutive sub-modules of vector fields}\label{sec:involutive}

In this section, we provide a further uniqueness result, based on the condition that the vector field $b(t,x)$ belongs to the $\Liph$-module generated by an involutive sub-algebra of $\cV_1 \subseteq \Lie(\G)$.

More precisely, let $\cS \subseteq \cV_1$ be an involutive sub-algebra, i.e.\ $[\cS, \cS] \subseteq \cS$, which, taking into account $[\cS, \cS] \subseteq \cV_2$, entails that $\cS$ is commutative, i.e.\
\begin{equation}\label{eq:submodule-commutative} [\cS, \cS] = 0.\end{equation}
Let $r = \dim \cS$, and let $Y_1, \ldots, Y_{r} \in \Lie(\G)$ be a basis of $\cS$. Our choice of homogeneous coordinates gives
\[ 
S = \exp(\cS) = \span \cur{v_1, \ldots, v_r} \subseteq \G,
\]
where $v_i = Y_i(0)$. Furthermore, due to \eqref{eq:submodule-commutative}, for every $s_1$, $s_2 \in \exp(\cS)$ we have
\[
 s_1 \cdot s_2 = s_1 + s_2,
\]
where $\cdot$ stands for the group operation, and the symbol $+$ denotes the usual sum of vectors. In particular, $S$ is a linear subspace of $\G$. As a useful result, we notice that
for each $j \in \cur{1, \ldots, r}$ there holds
\begin{equation}\label{eq:crucial-identity-commutative} Y_j(s) =d \ell_s(Y_j(0))  =  d \ell_s(v_j) =\frac{d}{dt}{\Big|_{t=0}} \bra{ s\cdot( t v_j) } = \frac{d}{dt}{\Big|_{t=0}} \bra{ s+ ( t v_j) } = v_j, \quad \text{ $ \forall s \in S$.}
\end{equation}
In other words, the restriction of $Y_j$ to $S$ is a constant vector field.


\begin{The}\label{thm:uniqueness-involutive}
Let $\cS$ satisfy \eqref{eq:submodule-commutative}, with basis $Y_1, \ldots, Y_{r} \in \Lie(\G)$, and let $b \in \Liph(H\Omega)$ be in the form
\begin{equation}\label{eq:b-involutive} b(x) = \sum_{i=1}^r a_i(x,t) Y_i(x), \quad \text{for $(t,x) \in [0,T]\times \Omega$,}\end{equation}
for $a_i \in \Liph(\Omega)$, for $i\in \cur{1,\ldots, r}$. Then, for every $x_0 \in \Omega$, there exists a unique solution to the Cauchy problem \eqref{eq:cauchy-problem} in $\Omega$.
\end{The}

The key of the proof relies on the following lemma.

\begin{Lem}\label{lem:key-uniqueness-involutive}
Let $\cS$ satisfy \eqref{eq:submodule-commutative}, with basis $Y_1, \ldots, Y_{r} \in \Lie(\G)$ and let $b \in C(H\Omega)$, be of the form \eqref{eq:b-involutive}, for $a_i \in C(\Omega)$. Then every $x_0 \in \Omega$, any solution to the problem \eqref{eq:cauchy-problem} in $\Omega$ belongs to $(x_0S) \cap \Omega$.
\end{Lem}

\begin{proof}
Without any loss of generality, we may assume $x_0 = 0$, (it is sufficient to apply Lemma~\ref{lem:change-initial-point-general} with $\bar{x} = x_0^{-1}$) and that $t \mapsto \gamma(t)$ is defined on some interval $[0,T]$. Since $\gamma$ is continuous, we have $\gamma(t) \in K$, for some compact $K$, for every $t \in [0,T]$, and in particular the functions $t\mapsto a_i(\gamma(t))$ are uniformly bounded. 

We introduce an auxiliary scalar product $\ang{\cdot, \cdot}$ on $\G$ (seen as the vector space $\R^{\q}$), and let $\sigma$ be the induced distance function from $S$, namely
\[ 
\sigma(x) := \inf_{s \in S} \abs{x - s}, \quad \text{for $x \in \G$,}
\]
and we denote by $\Sigma (x)\in S$ the minimum point, i.e., we let  $x \mapsto \Sigma (x)$ be the (linear) orthogonal projection on $S$, which satisfies $\ang{s, x - \Sigma (x) } = 0$, for every $s \in S$, hence 
\[
\sigma(x) = \abs{ (I - \Sigma) x }.
\]
The function $t  \mapsto u(t) := \sigma(\gamma(t))^2/2$ is differentiable on $[0,T]$, with derivative
\begin{equation*}\begin{split} \frac{d}{dt} u(t) = \frac {d}{dt} \frac{ \sigma(\gamma(t))^2}{2}  & = \ang{ (I - \Sigma)\dot \gamma (t) , (I - \Sigma)\gamma(t) }\\
 & = \ang{ \dot{\gamma}\bra{t}, (I - \Sigma)\gamma(t)} \\
& =\sum_{i =1}^r a_i(\gamma(t)) \ang{ Y_i(\gamma(t)), (I-\Sigma)\gamma(t)}\\
  & =\sum_{i=1}^r a_i(\gamma(t)) \ang{ Y_i(\gamma(t)) - Y_i(\Sigma(\gamma(t))) + Y_i(\Sigma(\gamma(t))), (I-\Sigma)\gamma(t)} \\
\text{(by \eqref{eq:crucial-identity-commutative} with $s= \Sigma(\gamma(t))$)} & =\sum_{i=1}^r a_i(\gamma(t)) \ang{ Y_i(\gamma(t)) -  Y_i(\Sigma(\gamma(t))) + v_i, (I-\Sigma)\gamma(t)}\\
&  =\sum_{i=1}^r a_i(\gamma(t)) \ang{ Y_i(\gamma(t)) -  Y_i(\Sigma(\gamma(t))), (I-\Sigma)\gamma(t)},
\end{split}\end{equation*} 
where, in the last identity, we used the fact that $v_i \in S$, hence $\ang{ v_i, (I-\Sigma) \gamma(t)}=0$.
 
Since each $Y_i$ is locally Lipschitz (with respect to the Euclidean structure) and $a_i(\gamma(t))$ is uniformly bounded, we obtain the inequality
\[ \abs{ \frac{d}{dt} u(t)} \le \sum_{i=1}^r c \abs{ (I-\Sigma)\gamma(t)}^2 \le  2rc u(t), \quad \text{for $t \in (0,T)$,}\]
hence, by Gronwall inequality,  $u(t) \le e^{2rc t} u(0)$, and since $u(\gamma(0))=0$, we have $u(\gamma(t)) = 0$ as well, for every $t \in [0,T]$.
\end{proof}

\begin{proof}[Proof of Theorem~\ref{thm:uniqueness-involutive}]
We may assume that $x_0 = 0$. By the lemma just proved, it is also sufficient to show uniqueness among the solutions that belong to $S\cap \Omega$. It is then immediate that the homogeneous norm \eqref{eq:smooth-gauge}, when restricted to $S \times S$, coincides with (a norm equivalent to) the Euclidean one, and that the group operation in $S$ is the usual sum of vectors: from this we conclude that uniqueness must hold by standard Cauchy-Lipschitz theory of ODE's in Euclidean spaces.
\end{proof}

\section{A counterexample to uniqueness}\label{sec:counterexample}

In this section we provide an example of a Lipschitz horizontal vector field for which the Cauchy problem \eqref{eq:cauchy-problem}  lacks of uniqueness, for suitably prescribed $x_0$.

On the Heisenberg group $\H = \R^3$, with the notation introduced in Section~\ref{sec:intro}, in particular \eqref{eq:heisenberg-operation}, \eqref{eq:heisenberg-norm} and \eqref{eq:heisenberg-fields}, we study a time-dependent, Lipschitz horizontal vector field. In Section~\ref{sec:autonomous-counterexample}, we also give a variant of it which is autonomous. We let 
\begin{equation}\label{eq:b-counterexample}b(t,x) := X_1(x) + a(t,x) X_2(x), \quad \text{for $(t,x) \in [0,T] \times \H$,} \end{equation}
where
\[ a(t,x) := \norm{ (t,0,0)^{-1} x } = \bra{ \bra{|x_1-t|^2 + |x_2|^2}^2 + |x_3 - t x_2|^2 }^{1/4}.\]
The triangle inequality of $\norm{\cdot}$ with respect to the group operation implies that $a(t,\cdot)$ is uniformly 1-Lipschitz continuous. Clearly, the curve $t \mapsto (t,0,0)$ is a solution to the associated differential equation  with $\gamma(0) = 0$. We aim to prove that this solution is not unique.
To see this, we begin by writing explicitly the system of differential equations as a special case of \eqref{eq:ode-system-componentwise}:
\begin{equation}\label{eq:counterexample-componentwise}
\left\{\begin{array}{ll}
\dot \gamma_1(t)=1 \\
\dot \gamma_2(t)=a(t,\gamma(t)) \\
\dot \gamma_3(t)=- \gamma_2(t) + a(t,\gamma(t))\gamma_1(t), \\
\end{array}\right.
\end{equation}
from which we see immediately that any solution must satisfy $\gamma_1(t) = t$. Thus, we may reduce ourselves to the study of the following $2$-dimensional problem
\begin{equation*}
\left\{\begin{array}{ll}
\dot \gamma_2(t)=a(t,(t, \gamma_2(t), \gamma_3(t)) \\
\dot \gamma_3(t)=- \gamma_2(t) + a\lls t,(t, \gamma_2(t), \gamma_3(t))\rrs t, \\
\end{array}\right.
\end{equation*}
with $\gamma_2(0) = \gamma_3(0)=0$, and provide existence of a non-trivial solution, i.e., different from $\gamma_2(t) = \gamma_3(t) = 0$, for $t \in [0,T]$. To this aim, we perform a suitable ``change of coordinates'' and introduce new variables
\[ u(t):= 18\gamma_2(t)t^{-3}, \quad v(t):= 36\bra{\gamma_2(t)t^{-3} - \gamma_3(t)t^{-4}},\]
along which the original system has a simpler analytic expression. Precisely, lack of uniqueness follows from a combination of the next two lemmas.

\begin{Lem}\label{lem:change-coordinates-counterexample}
Let $(u,v)\in C([0,T]; \R^2)\cap C^1((0,T);\R^2)$ be an integral solution to the following (singular) system
\begin{equation}\label{eq:counterexample-good-coordinates}
\left\{\begin{array}{ll}
 t \dot u  =  - 3 u   + 3 \bra{ (t/3)^{4} u^4 +  v^2 }^{1/4}\\
 t \dot v  =-  4 v   + 4 u 
  \end{array}\right.
\end{equation}
such that $u(0) = v(0) = 1$.
Then the curve
\[ \gamma(t)= (\gamma_1(t), \gamma_2(t), \gamma_3(t))  := \bra{t, \frac{t^3}{18} u(t), \frac{t^4}{36} (2u(t) - v(t))}, \quad \text{$t \in [0,T]$,}\]
solves the ODE \eqref{eq:counterexample-componentwise}, with $\gamma(0) = 0$.
\end{Lem}

\begin{The}\label{lem:existence-singular-ode}
There exists a solution $(u,v) \in C([0,T]; \R^2)$ to \eqref{eq:counterexample-good-coordinates} with $u(0) = v(0) = 1$.
\end{The}

\begin{proof}[Proof of Lemma~\ref{lem:change-coordinates-counterexample}]
For any $t >0$, we have
\begin{equation*}\begin{split} 
\dot \gamma_2 (t) = \frac{d}{dt}\pa{ \frac{t^3}{18} u(t)} & = \frac{t^2}{6} u(t) + \frac{t^2}{18}\bra{  - 3 u   + 3 \bra{ (t/3)^{4} u^4 +  v^2 }^{1/4}} \\
 & =  \frac{t^2}{6}\bra{ (t/3)^{4} u^4 +  v^2 }^{1/4}=  \bra{ \bra{\frac{t^3}{18} u}^{4} + \bra{ \frac{t^4}{36} v}^2 }^{1/4}\\
 &  = \bra{ \gamma_2^4(t) + \bra{ t\gamma_2 (t) - \gamma_3(t) }^2 }^{1/4} = a(t, \gamma(t) )
\end{split}\end{equation*}
and similarly
\begin{equation*}\begin{split} 
\dot \gamma_3 (t) = \frac{d}{dt}\pa{ \frac{t^4}{36} (2u(t) - v(t)) } & = \frac{d}{dt}\bra{  t \gamma_2(t) - \frac{t^4}{36} v(t) } \\
& = \gamma_2(t) +t \dot \gamma_2 (t)  - \frac{t^3}{9} v(t) + \frac{t^3}{9} v(t) -\frac{t^3}{9} u(t)\\
& =   \gamma_2(t) + t a(t, \gamma(t))  - 2 \gamma_2(t) = -\gamma_2(t) + t a(t, \gamma(t)).\qedhere
\end{split}\end{equation*}
\end{proof}
Before we proceed with the proof of Theorem~\ref{lem:existence-singular-ode},  we need
two comparison results.

\begin{Lem}\label{lem:comparison}
Let $\veps,\tau>0$ and let $z\in C([0,\tau]; \R)$ be everywhere differentiable in $(0,\tau)$ such that for $c_1$, $c_2$ $c_3$, $c_4 \in [0,\infty)$, $c_2>0$, we have
\[ (t+\veps) \dot z(t)\le c_1 - c_2 z (t)+ (t+\veps) ( c_3 z + c_4)\quad \text{for $t \in (0, \tau)$}\]
and $z(0) = c_1/c_2$. Then, one has
\[ z(t) \le \bra{\frac { c_1} {c_2} + c_4 (t+\veps)} e^{c_3(t+\veps)}\quad \text{for $t \in [0,\tau]$.}\]
\end{Lem}
\begin{proof}
Define $w(t) := z(t)e^{-c_3 (t+\veps) } - c_4 (t+\veps)$, which satisfies the 
following differential inequality, for $t \in (0,\tau)$,
\begin{equation*}
\begin{split}
(t+\veps)\dot w(t) & = (t+\veps)\dot z(t) e^{-c_3 (t+\veps) } - (t+\veps) c_3 e^{-c_3 (t+\veps) }z(t) - c_4(t+\veps) \\
&\le \bra{c_1 - c_2 z(t)} e^{-c_3 (t+\veps) } + (t+\veps) ( c_3 z + c_4)e^{-c_3 (t+\veps) } - (t+\veps) c_3 z e^{-c_3 (t+\veps)}  - c_4(t+\veps)\\
&= \bra{c_1 - c_2 z(t)} e^{-c_3 (t+\veps) } + (t+\veps) c_4 \bra{ e^{-c_3 (t+\veps)}-1} \\
& \le  c_1 - c_2 \bra{ w(t)+ c_4(t+\veps)} + (t+\veps) c_4 \bra{ e^{-c_3 (t+\veps)}-1} \\
& \le  c_1 - c_2 w(t),
\end{split}
\end{equation*}
and $w(0) = e^{-c_3 \veps}  c_1 / c_2 - c_4\ep \le  c_1/ c_2$. 
Observing that the solution of the Cauchy problem
\begin{equation*}
\left\{\begin{array}{ll}
\ds\dot y=\frac{c_1}{t+\ep}-\frac{c_2}{t+\ep} y \\
y(0)=c_1/c_2 \\
\end{array}\right.
\end{equation*}
on $[0,\tau]$ is the constant function $y\equiv c_1/c_2$, by standard comparison for ODE's we get $w(t) \le  c_1/ c_2$ for $t \in [0,\tau]$ and our claim follows.
\end{proof}

\begin{Lem}\label{lem:comparisonneg}
Let $\veps,\tau>0$ and let $w\in C([0,\tau]; \R)$ be everywhere differentiable in $(0,\tau)$ such that for $c_2$ $c_4 \in [0,\infty)$, $c_2>0$ and $c_1\le0$, we have
\beq\label{eq:diffc30}
 (t+\veps) \dot w(t)\le c_1 - c_2 w (t)+ c_4 (t+\veps) \quad \text{for $t \in (0, \tau)$}
\eeq
and $w(0) = c_1/c_2$. Then, one has
\[
 w(t) \le \frac { c_1} {c_2} + c_4 (t+\veps)  \quad \text{for $t \in [0,\tau]$.}
 \]
\end{Lem}
\begin{proof}
We fix $\delta>0$ and define $u=w+\delta$, so that the differential inequality \eqref{eq:diffc30} gives
\[
(t+\ep)\dot u(t)\le c_1+c_2\delta-c_2u+c_4(t+\ep).
\]
We choose $\delta>|c_1|/c_2$, hence $c_1+c_2\delta>0$.
Taking into account that $u(0)=(c_1+c_2\delta)/c_2$, we apply 
Lemma~\ref{lem:comparison}, that yields
\[
u(t)\le\frac{c_1+c_2\delta}{c_2}+c_4(t+\ep)
\] 
for all $t\in[0,\tau]$, that immediately leads us to our claim.
\end{proof}

\begin{proof}[Proof of Theorem~\ref{lem:existence-singular-ode}]
To show existence we argue by compactness (a standard strategy in this type of problems \cite{rachocircunkova_solvability_2008}) providing a-priori estimates uniform in $\veps \in (0,1)$ on the solutions $u^\veps$, $v^\veps$ to the system
\begin{equation}\label{eq:system-nice-form}
\left\{\begin{array}{ll}
 (t +\veps)\dot u^\veps  =  - 3 u^\veps   + 3 \bra{ (t/3)^{4} (u^\veps)^4 +  (v^\veps)^2 }^{1/4}\\
 (t+\veps) \dot v^\veps  =-  4 v^\veps   + 4 u^\veps,
  \end{array}\right.
\end{equation}
with  $u^\veps(0)=v^\veps(0)=1$. Such a system has locally (Euclidean) Lipschitz coefficients on an open set of $\R^2$, containing the initial datum, hence by standard Cauchy-Lipschitz theory it is well-posed on a maximal interval $[0,T_\veps)$. Part of the proof requires to show that $T_\veps$ is uniformly bounded from below as $\veps \downarrow 0$.

To simplify our notation, we drop the superscript $\veps$ and simply write $u$ in place of $u^\veps$, $v$ in place of $v^\veps$ in what follows. 
From the first equation of \eqref{eq:system-nice-form}, we deduce the inequality
\begin{equation}
\label{eq:lower-bound-easy}
(t+\veps)\dot u  =  - 3 u   + 3\bra{ (t/3)^{4} u^4 +  v^2 }^{1/4} \ge - 3 u,
\end{equation}
which by comparison (e.g.\ by Lemma~\ref{lem:comparison} applied with $z=-u$)  entails $u (t) \ge 0$, for $t \in [0,T_\veps)$. As a consequence,
\[ (t+\veps) \dot v  = -  4 v   + 4 u  \ge -  4 v\]
and again by comparison we have also $v(t) \ge 0$ for $t \in [0,T_\veps)$. More precise estimates can be given as follows.

\emph{Upper bounds.} By the inequalities $(a+b)^{1/4} \le a^{1/4} + b^{1/4}$, $a^{1/2} \le (1+a)/2$ valid for non-negative $a$'s and $b$'s, it follows  that
\begin{equation}\label{eq:upper-bound} (t+\veps)\dot u  = - 3 u   + 3\bra{ (t/3)^{4} u^4 +  v^2 }^{1/4} \le  - 3 u  + 3{v }^{1/2}  + t u\le  - 3 u   + \frac 3 2 \bra{v+1} + t u.\end{equation}
For any $\lambda \in \R$, from the second equation in \eqref{eq:system-nice-form}, we obtain
\[
(t+\veps)( \dot u + \lambda \dot v) \le (- 3 +4 \lambda) u   + \bra{\frac 3 2- 4 \lambda} v  + \frac 3 2 + t u. 
\]
We let $\lambda_1=1/2$ be the positive root of
\[
\lambda (- 3 +4 \lambda) = \bra{\frac 3 2- 4 \lambda}, \qs{namely} 8 \lambda^2 +2 \lambda - 3 = 0.
\]
As a consequence, we get
\[
(t+\veps)( \dot u + \lambda_1 \dot v) \le\frac 3 2- \bra{u   +\lambda_1 v}   + t u. 
\]
Since $v \ge 0$, we have
\[
(t+\veps)( \dot u + \lambda_1 \dot v)\le\frac 3 2- \bra{u+\lambda_1v}  + (t+\ep) \bra{u   +\lambda_1 v}. 
\]
By Lemma \ref{lem:comparison}, with $z =  u + \lambda_1  v$, $c_1 = \frac 3 2$, $c_2 =c_3=1$ and $c_4=0$, we have the upper bound
\begin{equation}
\label{eq:upper-bound-u+v}
 u + \lambda_1  v \le \frac 3 2 e ^{t+\veps}\quad \text{for $t \in [0,T_\veps)$}
\end{equation}
and in particular both $u$ and $v$ are non-negative and uniformly bounded in every compact set of $\R^+$. It follows that $T_\veps =+\infty$ for every $\veps>0$.
Actually, since for $0<\ep<1$ and $0\le t\le 2-\ep$, there exists a constant $c>1$, independent
of $\ep$, such that
\beq\label{eq:expcc}
\frac{3}{2} e^{t+\ep}\le \frac{3}{2}+c (t+\ep),
\eeq
then, taking into account the equation for $v$ in \eqref{eq:system-nice-form}, we have
\[ (t+\veps) \dot v  = -  4 v   + 4 u  = -  4 (1+\lambda_1)v   + 4( u +\lambda_1 v) \le -  4 (1+\lambda_1)v  + 6 + c(t+\veps)
\]
for all $t\in[0,2-\ep]$.
By Lemma~\ref{lem:comparison}, for every $t \in [0, 1]$ we have 
\[
v(t) \le 1+c(t+\veps).
\]

\emph{Lower bounds.} In particular, there holds $v(t) - c(t+\veps) \le 1$ for $t \in [0,1]$. 
We now restrict the choice of $\ep\in (0,1/c)$.
We set $\tau_\ep=1$ if $v(t)>c(t+\ep)$ for all $t\in[0,1]$. Otherwise, $0<\tau_\ep<1$
is the minimum time such that $v(\tau_\ep) = c(\tau_\ep+\veps)$.
Since $\tau_\veps$ may depend on $\ep$, we wish to show that $\tau_\ep$ is bounded away from 
zero, uniformly with respect to $\veps \downarrow 0$. Indeed, for 
$t \in [0, \tau_\veps]$, we have
\begin{equation}
\label{eq:lower-bound}
\bra{ v(t)}^{1/2} \ge  \bra{ v(t)-c(t+\veps)}^{1/2} \ge v(t)-c(t+\veps),
\end{equation}
since $0 \le v(t) - c(t+\veps) \le 1$. Thus, for $t \in (0, \tau_\veps )$, we have the differential lower bound, from the equation for $u$ in \eqref{eq:system-nice-form},
\[(t+\veps)\dot u  \ge  - 3 u + 3 v -3c(t+\veps). \]
If we add to this differential inequality the equation for $\lambda v$, for fixed $\lambda \in \R$, we obtain
\[(t+\veps)\bra{\dot u + \lambda \dot v} \ge  (- 3 + 4 \lambda)u + (3-4\lambda) v -3c(t+\veps). \]
Choosing $\lambda_2 = 3/4$ in the previous inequality, we deduce
\[
 (t+\veps)\bra{\dot u + \lambda_2 \dot v} \ge -3c(t+\veps)
\]
for $t \in (0, t_\veps)$. In particular, $\bra{\dot u + \lambda_2 \dot v} \ge -3c$
and this implies the lower bound 
\beq\label{eq:-3c}
 u +  \lambda_2 v \ge 1 + \lambda_2 -3ct\ge  1 + \lambda_2 -3c(t+\ep),
\eeq
which we use to find the following lower bound
\[
 (t+\veps) \dot v  = -  4 v   + 4 u  = -  4 (1 +\lambda_2)v   + 4( u +\lambda_2 v) \ge 4(1 + \lambda_2) -  4 (1+\lambda_2)v  -12\, c(t+\veps).
\]
Now, we apply Lemma \ref{lem:comparisonneg} to $w=-v$, getting $v \ge 1 - 12\,c(t+\veps)$ for $t \in [0, \tau_\veps]$. In particular, we surely have that $\tau_\veps>0$ is bounded from below by $\frac{1}{13\,c} - \veps$. If we let $\ep$ vary in $\ep\in(0,1/26\,c)$, we get the uniform estimate
\[
\tau_\ep\ge\frac{1}{26\,c}.
\]

\emph{Compactness in $C([0,\tau]; \R^2)$}. For $\veps \in(0,1/26\,c)$ and setting
$\tau=1/26\,c$, we have the bound $\abs{v(t) - 1} \le 12\,c(t+\veps)$ for all $t \in [0,\tau]$.
Joining the estimates \eqref{eq:upper-bound-u+v}, \eqref{eq:expcc} and \eqref{eq:-3c},
for every $t \in [0,\tau]$, we get 
\[
\abs{u(t)-1} \le 12\, c(t+\veps).
\]
If we set $\tilde c=12 c$ and plug the previous bounds in the system for $u$ and $v$, we deduce that
\begin{equation}\label{eq:compactness}
 (t+\veps) \abs{ \dot v}  \le \tilde c(t+\veps), \quad (t+\veps) \abs{\dot u}  \le \tilde c(t+\veps), \quad \text{for $t \in (0,\tau)$,}
 \end{equation}
hence $u$ and $v$ are uniformly Lipschitz continuous on $[0,\tau]$ with respect to
$\ep\in(0,1/26\,c)$. 
By Ascoli-Arzel\`a theorem, we may extract a family $u^{\veps(n)}$, $v^{\veps(n)}$ converging in $C([0,\tau],\R^2)$ towards continuous functions $u$, $v$. These functions solve the system \eqref{eq:counterexample-good-coordinates}, with initial conditions $u(0) = 1$, $v(0) = 1$. Indeed, if we write the integral form of the approximating systems, for $t \in [0,\tau]$,
\begin{equation*}
\begin{split}
 u^{\veps}(t) &= 1 + \int_0^t \sqa{- 3 u^\veps (s)  + 3 \bra{ (s/3)^{4} (u^\veps(s))^4 +  (v^\veps(s))^2 }^{1/4}} \frac{ds}{s+\veps},\\
 v^{\veps}(t) &= 1 + \int_0^t\sqa{-  4 v^\veps(s)   + 4 u ^\veps(s)} \frac{ds}{s+\veps},
\end{split}
\end{equation*}
we have pointwise convergence of the integrands as $\veps>0$ (everywhere except for $s = 0$), uniformly dominated by some constant, and the limit by Lebesgue theorem.
\end{proof}

\subsection{An autonomous counterexample}\label{sec:autonomous-counterexample} To exhibit a counterexample in the autonomous case, it is sufficient to slightly modify the vector field $b$  in \eqref{eq:b-counterexample}, considering instead
\[
 a(x) :=  \inf_{s \in \R} d((s,0,0), x ), \quad \text{for $x \in \H$,}
\]
which is still a horizontal Lipschitz function, with respect to the distance $d$, and it is everywhere less or equal than $d((t,0,0), x )$. The proof goes almost identically as in the previous case: in particular, we perform the same change of variables.
As a result, from $\gamma_1$, $\gamma_2$, $\gamma_3$, taking into account that $\gamma_1(t) = t$, we obtain $u$, $v$, along with the system
\begin{equation*}
\left\{\begin{array}{ll}
 t \dot u  =  - 3 u   + 3 F(t,u,v)\\
 t \dot v  =-  4 v   + 4 u 
  \end{array}\right.
\end{equation*}
with  $u(0)=v(0)=1$ and $F(t,u,v) = \frac { 6}{t^{2}} a\bra{t,t^3 u/ 18, t^4(2u-v)/36)} $. 
The existence of a solution to the previous singular system with the given initial conditions leads to our counterexample. We use a compactness argument for the approximating systems
\begin{equation}\label{eq:approximating-systems-autonomous-case}
\left\{\begin{array}{ll}
 \bra{t +\veps} \dot u^\veps  =  - 3 u^\veps   + 3 F(t,u^\veps,v^\veps)\\
  \bra{t +\veps} \dot v ^\veps =-  4 v^\veps   + 4 u^\veps 
  \end{array}\right.
\end{equation}
with  $u^\veps(0)=v^\veps(0)=1$, as $\veps \downarrow 0$. For every $\veps>0$, existence (up to some maximal time $T_\veps>0$) for $(u^\veps, v^\veps)$ -- that we denote simply by $(u,v)$ in what follows --  is a consequence e.g.\ of Peano's theorem and the fact that $F$ (initially defined on $(0,\infty)\times \R^2$) can be uniquely extended to a continuous function with respect to $(t,u,v) \in [0,\infty)\times \R^2$. Indeed, one has the identities
\begin{equation}\label{eq:autonomous-case-identities} \begin{split}  F(t,u,v) & = \frac { 6}{t^{2}} \inf_{s \in \R}  \bra{ \bra{|s-t|^2 + |t^3 u/ 18|^2}^2 + |t^4(2u-v)/36 - s t^3 u/ 18 |^2 }^{1/4} \\
\text{$\sqa{s \mapsto ts}$}& =\frac { 6 }{t} \inf_{s \in \R}  \bra{ \bra{|s-1|^2 + t^4 | u/18|^2}^2 + t^4 |(2u-v)/36 - s u/ 18 |^2 }^{1/4}\\
\text{$\sqa{s \mapsto s+1}$} & =\frac { 6 }{t} \inf_{s \in \R}  \bra{ \bra{ s^2 + t^4 | u/18|^2}^2 + t^4 |v/36 + s u/ 18 |^2 }^{1/4}\\
\text{$\sqa{s \mapsto ts}$} & = 6 \inf_{s \in \R}  \bra{ \bra{ s^2 + t^2 | u/18|^2}^2 +  |v/36 + s t u/ 18 |^2 }^{1/4},\\
\end{split}
\end{equation}
which shows that $F(t,u,v)$ can be defined also when $t=0$. To show continuity, we use the fact that, for fixed $a$, $b$, $c\in \R$ with $a\ge 0$, the convex, non-negative, function $s \mapsto f(s) := \bra{ s^2 + a}^2 + \bra{bs+c }^2$ on $\R$ admits a unique minimum point $\bar{s}$ and minimum value $f(\bar{s})$, and it is not difficult to show that $(a,b,c) \mapsto \bar{s}(a,b,c)$ is continuous. Hence, the composition $(t,u,v) \mapsto f(\bar{s}\bra{t^2| u/18|^2, t u/ 18, v/36 })$ is continuous, and so its fourth root, which coincides with $F(t,u,v)/6$.

Next, we notice that the simple inequalities  $0 \le F(t,u,v) \le \bra{ (t/3)^{4} u^4 +  v^2 }^{1/4}$ hold, the latter obtained letting $s=0$ in the last line of \eqref{eq:autonomous-case-identities}. This allows us to argue analogously as in the proof of Theorem~\ref{lem:existence-singular-ode}, obtaining the analogue of \eqref{eq:lower-bound-easy} and the lower bounds $u(t)\ge 0$, $v(t) \ge 0$, as well as the analogue of \eqref{eq:upper-bound}, hence \eqref{eq:upper-bound-u+v} and 
\begin{equation}\label{eq:upper-bound-v} v(t) \le 1  + c_1(t+\veps),\end{equation}
for some constant $c_1>0$. In particular, we have $T_\veps > 1$, hence some solution $(u, v)$ to \eqref{eq:approximating-systems-autonomous-case} exists on the interval $[0,1]$, for every $\veps >0$, with both $\norm{u} := \sup_{t \in [0,1]} \abs{u(t)} \le c_2$ and $\norm{v} := \sup_{t \in [0,1]} \abs{v(t)} \le c_3$, uniformly bounded as $\veps \downarrow 0$ (i.e., $c_2>0$ and $c_3>0$ are fixed numbers). Therefore, because of the continuity of the minimum point $(a,b,c) \mapsto \bar{s}(a,b,c)$, we have that the function $t\mapsto \sigma(t):=\bar{s}\bra{t^2| u/18|^2, t u/ 18, v/36 }$ is uniformly bounded on $[0,1]$, i.e.\ $\norm{\sigma} = \sup_{t \in [0,1]} \abs{\sigma(t)}$ is bounded from above by some function of $\norm{u}$ and $\norm{v}$, and ultimately by some constant $c_4>0$ (independent of $\veps \downarrow 0$).

Let then $\tau_\veps>0$ be the first time such that $v(t) - (t+\veps) c_5 =0$, where $c_5$ is a fixed constant, larger than constant $c_1$ appearing in \eqref{eq:upper-bound-v} and than $2 c_2c_4$ (such a constant $c_5$ exists, provided that $\veps>0$ is small enough, since $v(0) =1$). As in the autonomous case, $\tau_\veps$ depends upon $\veps>0$ and in principle it could be that $\tau_\veps \to 0$ as $\veps \downarrow 0$, but the argument that we give provides a uniform lower bound for $\tau_\veps$. Indeed, the crucial point is to show, in place of \eqref{eq:lower-bound}, the validity of the inequality
\begin{equation}\label{eq:crucial-lower-bound-autonomous-counterxample} F(t,u(t),v(t)) \ge v(t) - c_5(t+\veps), \quad \text{for $t \in [0,\tau_\veps \land 1]$.} \end{equation}
For $t \in [0,\tau_\veps \land 1]$, we have
\begin{equation*}\begin{split}  F(t,u(t),v(t)) & =  6 \inf_{s \in \R}  \bra{ \bra{ s^2 + t^2 | u(t)/18|^2}^2 +  |v(t)/36 + s t u(t)/ 18 |^2 }^{1/4}\\
\bra{\text{by definition of $\sigma$}} & =  6 \bra{ \bra{ \sigma(t)^2 + t^2 | u(t)/18|^2}^2 +  |v(t)/36 + \sigma(t) t u(t)/ 18 |^2 }^{1/4}\\
& \ge  6  \abs{ \bra{v(t) + 2 t \sigma(t) u(t)}/ 36 }^{1/2}\\
& = \abs{v(t) + 2 t \sigma(t) u(t)}^{1/2}\\
& \ge \abs{ v(t) - v(t)  - 2 (t+\veps) \norm{\sigma} \norm{u}}^{1/2} \ge \abs{ v(t) - c_5(t+\veps) }^{1/2} 
\end{split}
\end{equation*}
where in the last inequality we used the fact that
\[ v(t) + 2 t \sigma(t) u(t) \ge v(t)  - 2 (t +\veps)\norm{\sigma} \norm{u} \ge  v(t)  - c_5(t+\veps) \ge 0, \text{for $t \in [0, \tau_\veps]$.}\]
Because of \eqref{eq:upper-bound-v} and the fact that $c_5 \ge c_1$, we have that $v(t) - c_5 t \le 1$, hence inequality \eqref{eq:crucial-lower-bound-autonomous-counterxample}: 
\[ F(t,u(t),v(t)) \ge  \abs{ v(t) - c_5(t+\veps) }^{1/2}  \ge v(t) - c_5(t+\veps), \text{for $t \in [0, \tau_\veps]$.} \]
As this lower bound is settled, we argue identically as in the non-autonomous case, obtaining lower bounds for $u$, $v$ and finally the desired inequalities analogous to \eqref{eq:compactness}, yielding compactness in $C([0,\tau], \R^2)$ as $\veps \downarrow 0$.

\bibliography{biblio-dario}
\bibliographystyle{plain}

\end{document}